\documentclass[11pt]{amsart}
\usepackage{mathrsfs}
\usepackage{amssymb}
\usepackage{setspace}
\usepackage{titletoc}
\usepackage{amscd}
\usepackage{amsmath, amssymb}
\usepackage{amsfonts}
\usepackage[colorlinks,linkcolor=blue,citecolor=blue, pdfstartview=FitH]{hyperref}

\usepackage{amssymb}
\usepackage{amsmath}
\usepackage{latexsym}
\usepackage{amsfonts}
\usepackage{graphicx}

\theoremstyle{plain}
\newtheorem{thm}{Theorem}[section]
\newtheorem{theorem}[thm]{Theorem}
\newtheorem{lemma}[thm]{Lemma}
\newtheorem{corollary}[thm]{Corollary}
\newtheorem{proposition}[thm]{Proposition}

\theoremstyle{definition}

\newtheorem{remark}[thm]{Remark}

\newtheorem{definition}[thm]{Definition}

\newtheorem{example}[thm]{Example}

\newtheorem{defn-thm}[thm]{Definition-Theorem}
\newtheorem{conjecture}[thm]{Conjecture}

\newcommand{\p}{{\partial }}
\newcommand{\bp}{{\overline{\partial }}}
\newcommand{\sq}{{\sqrt{-1}}}

\DeclareMathOperator{\Ric}{Ric}

\allowdisplaybreaks

\newcommand{\be}{\begin{equation}}
\newcommand{\ee}{\end{equation}}
\newcommand{\bea}{\begin{eqnarray}}
\newcommand{\eea}{\end{eqnarray}}

\newcommand{\eps}{\varepsilon}

\newcommand{\vs}{\vspace{0.5cm}}

\def\XXint#1#2#3{{\setbox0=\hbox{$#1{#2#3}{\int}$ }
\vcenter{\hbox{$#2#3$ }}\kern-.6\wd0}}

\usepackage{fancyhdr}
\begin{document}

\title{\ \ \ \ \ \ \ \ On real bisectional curvature \newline \ for Hermitian manifolds}
\author{Xiaokui Yang$^\S$}\thanks{$^\S$Research supported by China's
Recruitment
 Program of Global Experts and National Center for Mathematics and Interdisciplinary Sciences,
 Chinese Academy of Sciences.}
\address{Xiaokui Yang, Morningside Center of Mathematics and Hua Loo-Keng Key Laboratory of Mathematics, Academy of Mathematics and Systems Science\\ Chinese Academy of Sciences, Beijing, 100190, China}

 \email{xkyang@amss.ac.cn}

\author{Fangyang Zheng$^\dagger$}
\address{Fangyang Zheng. Department of Mathematics,
The Ohio State University, 231 West 18th Avenue, Columbus, OH 43210,
USA and Zhejiang Normal University, Jinhua 321004, Zhejiang, China.}
\email{{zheng.31@osu.edu}}
\thanks{$^\dagger$Research supported by a Simons Collaboration Grant from the Simons Foundation.}

\begin{abstract}
Motivated by the recent work of Wu and Yau on the ampleness of canonical line bundle for compact K\"ahler
manifolds with negative holomorphic sectional curvature, we introduce a new curvature notion called {\em real bisectional curvature}
for  Hermitian manifolds. When the metric is K\"ahler, this is just the holomorphic sectional curvature $H$,
and when the metric is non-K\"ahler, it is slightly stronger than $H$. We classify compact Hermitian manifolds
 with constant non-zero real bisectional curvature, and also slightly extend Wu-Yau's theorem to the Hermitian case.
  The underlying reason for the extension is that the Schwarz lemma of Wu-Yau works the same when the target metric
   is only Hermitian but has nonpositive real bisectional curvature.
\end{abstract}

\maketitle{\small{ \begin{spacing}{0.8}
 \tableofcontents\end{spacing}}}

\section{Introduction and the statement of results}

The study of holomorphic sectional curvature in K\"ahler geometry has been a classic topic,
and it also attracted some attention in recent years, for example, the work of Heier and collaborators on
 compact K\"ahler manifolds with positive or negative holomorphic sectional curvature (\cite{HW, HLW10, HLW14, AGZ}),
  and the recent breakthrough of Wu and Yau \cite{WY} in which they proved that any projective K\"ahler manifold with negative
   holomorphic sectional curvature must have ample canonical line bundle. This result was obtained by Heier et.\! al.\! earlier
    under the additional assumption of the Abundance Conjecture.

In \cite{TY}, Tosatti and the first named author proved that any compact K\"ahler manifold with nonpositive holomorphic sectional
 curvature must have nef canonical line bundle, with that in hand,  they were able to drop the projectivity assumption in the
  aforementioned Wu-Yau Theorem. More recently, Diverio and Trapani \cite{DT} further generalized the result by assuming that
   the holomorphic sectional curvature is only {\em quasi-negative,} namely, nonpositive everywhere and negative somewhere
    in the manifold. In \cite{WY1}, Wu and Yau give a direct proof of the statement that any compact K\"ahler manifold with
     quasi-negative holomorphic sectional curvature must have ample canonical line bundle.

In this direction, the following conjectures are still open:

\begin{conjecture}\label{conjecture1.1}
Let $M^n$ be a compact complex manifold.
 \begin{itemize}
\item[(a)] If $M$ is Kobayashi hyperbolic, then its canonical line
bundle $K_M$ is ample.
\item[(b)] If $M$ admits a Hermitian metric with quasi-negative holomorphic sectional curvature, then $K_M$ is ample.
\item[(c)] If $M$ admits a Hermitian metric with negative holomorphic sectional curvature, then $K_M$ is ample.
\end{itemize}
\end{conjecture}

By Yau's Schwarz Lemma \cite{YauS}, any Hermitian manifold with
holomorphic sectional curvature bounded from above by a negative
constant must be Kobayashi hyperbolic. So (a) implies (c). Clearly,
(b) also implies (c).

As an attempt to push the results of \cite{WY} to conjecture (b) or (c) above,  and also
 to study holomorphic sectional curvature on Hermitian manifolds in general, we introduce the following curvature term for Hermitian manifolds.

\begin{definition}
Let $(M^n,g)$ be a Hermitian manifold, and denote by $R$ the curvature tensor of the Chern
 connection. For $p\in M$, let $e= \{ e_1, \ldots , e_n\}$ be a unitary tangent frame at $p$,
  and let $a=\{ a_1, \ldots , a_n\} $ be non-negative constants with $|a|^2 = a_1^2 + \cdots + a_n^2 > 0$.
   Define the {\em real bisectional curvature} of $g$ by
$$ B_g(e,a)=\frac{1}{|a|^2} \sum_{i, j=1}^n R_{i\overline{i}j\overline{j}} a_ia_j.$$
\end{definition}

We will say that a Hermitian manifold $(M^n,g)$ has {\em positive real bisectional curvature,}
 denoted by $B_g>0$, if for any $p\in M$ and any unitary frame $e$ at $p$,
  any nonnegative constants $\{ a_1, \ldots , a_n\}$ (not all of them zero),
   it holds that $B_g(e,a) >0$. The term $B_g\geq 0$, $B_g<0$, or $B_g\leq 0$
    are defined similarly. We will say that $B_g$ is {\em quasi-positive,}
     if it is non-negative everywhere, and is positive at a point $p\in M$ for all choices of $e$ and  $a$.

Let $c$ be a constant. It is easy to see that, at any $p\in M$ and for any fixed unitary frame $e$,
 then $B_g(e',a) > c$  for any choice of unitary frame $e'$ and nonnegative (but not all zero) constants $a$ if and only if
\begin{equation}
\sum_{i,j,k, \ell } R_{i\overline{j}k\overline{\ell}} \xi_{ij} \xi_{k\ell } > c \ \mbox{tr}(\xi^2)
\end{equation}
for all non-trivial, nonnegative, Hermitian $n\times n$ matrix $\xi$. The condition $B_g \geq c$, $<c$, or $\leq c$ can be defined similarly.

Recall that the holomorphic sectional curvature in the direction $v$ is defined by $H(v)= R_{v\overline{v}v\overline{v}}/|v|^4$.
If we take $e$ so that $e_1$ is parallel to $v$, and take $a_1=1$, $a_2=\cdots =a_n=0$, then $B$ becomes $H(v)$.
So  $B>0$ ($\geq 0$, $<0$, or $\leq 0$) implies $H>0$ ($\geq 0$, $<0$, or $\leq 0$).
Also, if $B$ is quasi-positive or quasi-negative, then so is $H$. Using Berger's averaging trick,
 we will show in the next section that

\begin{proposition}\label{proposition1.3}
If $(M^n,g)$ is K\"ahler, or more generally, K\"ahler-like, then
$B_g > 0$  $(\geq 0$, $<0$, or $\leq 0)$ if and only if $H>0$ $(\geq
0$, $<0$, or $\leq 0)$. In the non-K\"ahler case, there are
Hermitian metrics $g$
 such that $H>0$ but $B_g \ngeq 0$, and there are also Hermitian  metrics $g$ such that $H<0$ but $B_g \nleq 0$.
\end{proposition}

Note that the term {\em K\"ahler-like} means that the Chern curvature tensor obeys all the symmetries of
 the curvature of K\"ahler metrics (see \cite{YZ}). The second part of the above proposition says that,
  in the non-K\"ahler case, real bisectional curvature is indeed a stronger curvature condition than
   holomorphic sectional curvature, and the concept is a natural generalization of holomorphic sectional
    curvature for K\"ahler manifolds. On the other hand, the difference between $B_g$ and $H$ is subtle
     and not very big, as we will see in the next section. For instance, the sign of $B_g$ does not
      control the sign of any of the three Ricci curvature tensors (see \S 2).

It is a natural question to ask when will a Hermitian manifold have constant real bisectional curvature.
 To this end, we have the following:

\begin{theorem}\label{Theorem1.4}
Let $(M^n,g)$ be a compact Hermitian manifold whose real bisectional
curvature is constantly equal to $c$. Then $c\leq 0$. Moreover, when
$c=0$, then $(M,g)$ is a balanced manifold with vanishing first,
second, and third Ricci tensors, and its Chern curvature satisfies
the property $R_{x\overline{y}u\overline{v}} = -
R_{u\overline{v}x\overline{y}} $ for any type $(1,0)$ complex
tangent vectors $x$, $y$, $u$, $v$.
\end{theorem}

\noindent We would like to propose the following conjecture:

\begin{conjecture}
Let $M^n$ ($n\geq 3$) be a compact Hermitian manifold with vanishing
real bisectional curvature $c$. Then $c=0$, and  its Chern curvature
tensor $R=0$.
\end{conjecture}

By Boothby's theorem, compact Hermitian manifolds with vanishing Chern curvature are precisely the
 compact quotients of complex Lie groups equipped with left invariant metrics.

Besides the constant real bisectional curvature cases, more generally, it would certainly be
 very interesting to try to understand the class of all compact Hermitian manifolds with positive
  (or negative) real bisectional curvature. For instance, we could raise the following

\begin{conjecture}\label{conjecture1.6}
Let $(M^n,g)$ be a compact Hermitian manifold, $B_g$ its real
bisectional curvature, and $K_M$ its canonical line bundle.
\begin{itemize} \item[(a)] If $B_g>0$, then $M$ is simply-connected.
\item[(b)] If $B_g>0$, then $M$ is rationally connected.
\item[(c)] If $B_g\leq 0$, then $K_M$ is nef.
\item[(d)] If $B_g$ is quasi-negative, then $K_M$ is ample.
\end{itemize}
\end{conjecture}

Of course the part (d) here is just a slightly weaker version of
Conjecture \ref{conjecture1.1}, part (b). In this paper, we will
focus on the negative/nonpositive real bisectional curvature cases,
and the main observation of this article is that  the Schwarz lemma
of Wu and Yau (\cite[Proposition~9]{WY}) can be generalized to
Hermitian manifolds when the target metric has negative real
bisectional curvature. As a consequence, the results of Wu-Yau
\cite{WY, WY1}, Tosatti-Yang \cite{TY}, and Diverio-Trapani
\cite{DT} can be partially generalized to the Hermitian case. To be
more precise, we have the following:

\begin{theorem}\label{Theorem1.7}
Let $(M,h)$ be a compact Hermitian manifold with nonpositive real
bisectional curvature. If $M$ is K\"ahlerian, then its canonical
line bundle is nef.
\end{theorem}

\begin{theorem}\label{Theorem1.8}
Let $(M,h)$ be a compact Hermitian manifold with quasi-negative real
bisectional curvature. If $M$ is K\"ahlerian, then its canonical
line bundle is ample.
\end{theorem}

Recall that K\"ahlerian means that the manifold admits a K\"ahler metric, which of course does not have to be $h$ here.
 It would certainly be highly desirable to drop this assumption, but at this point we have no idea how to achieve that goal.
  We do observe the following special case, which provides some partial evidence to part (d) of
  Conjecture \ref{conjecture1.6}.

\begin{theorem}\label{Theorem1.9}
Let $M$ be a compact Hermitian manifold with quasi-negative real
bisectional curvature.
 Let $N^n$ be a compact complex manifold which admits a holomorphic fibration $f: N \rightarrow Z$,
  where a generic fiber is a compact K\"ahler manifold with $c_1=0$. Then  $M^n$ cannot be bimeromorphic to $N^n$.
\end{theorem}

Also, as immediate consequences of the Hermitian version of Wu-Yau's Schwarz Lemma (see \S 4), we get the following rigidity results:

\begin{theorem}\label{Theorem1.10}
Let $(M,g)$ be a Hermitian manifold with nonnegative second Ricci
curvature and with bounded Gauduchon $1$-form $\eta$. Assume that as
a Riemannian manifold, it is complete and has Ricci curvature
bounded from below. Let $(N,h)$ be a
 Hermitian manifold with real bisectional curvature bounded from above by a negative constant. Then any holomorphic map from $M$ to $N$ must be constant.
\end{theorem}

\begin{theorem}\label{Theorem1.11}
Let $(M,g)$ be a compact Hermitian manifold with nonnegative second
Ricci curvature, and $(N,h)$ be a
 Hermitian manifold with nonpositive real bisectional curvature. Then any non-constant holomorphic map from $M$ to $N$ is totally geodesic.
\end{theorem}

\noindent\textbf{Acknowledgement.} We would like to thank Bing-Long
Chen, Gordon Heier, Kefeng Liu, Kai Tang, Valentino Tosatti, Damin
Wu, Hongwei Xu, Bo Yang,
 Shing-Tung Yau and Xiangyu Zhou for their interests and/or discussions.

\section{The real bisectional curvature of  Hermitian manifolds}

Let $(M^n,g)$ be a Hermitian manifold. Under a local holomorphic coordinate system $(z_1, \ldots , z_n)$,
 the curvature tensor of the Chern connection has components
\begin{equation}
 R_{i\overline{j}k\overline{\ell}} = - \frac{\partial^2 g_{k\overline{\ell}} } {\partial z_i \partial \overline{z}_j}
 + g^{p\overline{q}} \frac{\partial g_{k\overline{q}} } {\partial z_i }   \frac{\partial g_{p\overline{\ell}} } {\partial
 \overline{z}_j}.
 \end{equation}
Let $e$ be a unitary frame of type $(1,0)$ tangent vectors and $a_1, \ldots , a_n$ be
 non-negative constants with $|a|^2=a_1^2 + \cdots + a_n^2 >0$. Recall that the real bisectional curvature $B_g$
  in the direction of $e$ and $a$ is defined by
$$ B_g(e,a) = \frac{1}{|a|^2} \sum_{i,j=1}^n R_{e_i\overline{e}_ie_j\overline{e}_j}  a_i a_j .$$

\begin{remark}
Note that the definition of real bisectional curvature is somewhat analogous to the notion of
 {\em quadratic orthogonal bisectional curvature} defined in \cite{WYZ}, see also \cite{Chau-Tam, Chau-Tam1, FuWangWu, LiWuZheng, Zhang}.
  However, the two curvature notions are actually quite different, in the sense that the former is a slight
   generalization of holomorphic sectional curvature $H$ (and is actually equivalent to $H$ when the metric is K\"ahler)
    while the latter is closer to orthogonal bisectional curvature.
\end{remark}

For any type $(1,0)$ tangent vector $v\neq 0$, if we choose $e$ so that $e_1$ is parallel to $v$, and choose $a_1=1$, $a_2=\cdots =a_n=0$, then we get
$$ B_g(e,a)= R_{v\overline{v}v\overline{v}}/|v|^4 = H(v). $$
So the holomorphic sectional curvature is part of the real bisectional curvature, and the sign of $B$ guarantees the sign of $H$.

Conversely, let $\omega_{FS}$ be the Fubini-Study metric on ${\mathbb P}^{n-1}$ with unit volume, and let $[w_1:\cdots :w_n]$
 be the standard unitary homogeneous coordinate. Then it is well-known that
$$ \int_{{\mathbb P}^{n-1} } \frac{w_i\overline{w}_j w_k\overline{w}_{\ell}}{|w|^4} \omega_{FS}^{n-1} = \frac{\delta_{ij}\delta_{k\ell} + \delta_{i\ell} \delta_{kj} }{n(n+1)} ,$$
so if we fix a point $p\in M$ and fix any nonnegative constants $b_1, \ldots , b_n$,
 not all zero, then by considering the integration
$$ \sum_{i,j,k,\ell=1}^n \int_{{\mathbb P}^{n-1} } R_{i\overline{j}k \overline{\ell}} \frac{b_iw_i \overline{b}_j\overline{w}_j
  b_kw_k \overline{b}_{\ell}\overline{w}_{\ell} }{|w|^4} \omega_{FS}^{n-1} = \frac{2}{n(n+1)}\sum_{i,k=1}^n (R_{i\overline{i}k \overline{k}}
  + R_{i\overline{k}k \overline{i}}) b_i^2 b_k^2,$$
we know that, if $H>0$, then the real bisectional curvature $B_g>0$ if the metric $g$ is K\"ahler-like (\cite{YZ}),
 as in this case the two curvature terms in the right hand side of the above equality are equal. For a general Hermitian metric $g$,
  if $H>0$, then we know that, for any unitary frame $e$ and any nonnegative constants $a_1, \ldots , a_n$ with at least one of them being positive, it holds that
\begin{equation}
\sum_{i,k} (R_{i\overline{i}k \overline{k}} + R_{i\overline{k}k
\overline{i}}) \ a_ia_k > 0.
\end{equation}
This is an equivalent way to describe $H> 0$. It is analogous to the definition of $B_g>0$,
 but not exactly the same (when the metric is not K\"ahler-like in the sense of \cite{YZ}). So we have proved the first part of Proposition \ref{proposition1.3}.

To see the second part of Proposition \ref{proposition1.3}, let us
consider the following example:

\begin{example}\label{Example2.2}
On a ball centered at the origin and with small radius $D \subseteq
{\mathbb C}^n$ ($n\geq 2$), consider the $U(n)$-invariant
 Hermitian metric $g$ defined by
$$ g_{i\overline{j}} = (1+|z|^2) \ \delta_{ij} + (\eps -2) \ \overline{z}_iz_j $$
where $\eps \in (0,1)$ is a constant and $|z|^2=|z_1|^2+\cdots +|z_n|^2$, with $(z_1, \ldots , z_n)$ the standard Euclidean
 coordinate in ${\mathbb C}^n$. We claim that the metric has positive holomorphic sectional curvature, but the real bisectional curvature is not even nonnegative.
\end{example}

At the origin $z=0$, since all the first order derivatives of $g$ are zero, we get
$$ R_{i\overline{j}k \overline{\ell}} = -\frac{\partial^2    g_{k \overline{\ell}}  }{ \partial z_i \partial \overline{z}_j }
 = - \delta_{ij}\delta_{k\ell} + (2-\eps ) \delta_{i\ell}\delta_{kj}.$$
From this, we see that the holomorphic sectional curvature at the origin is constantly $H(v) = 1-\eps$, which is positive.
 On the other hand, if $e_{\alpha } = \sum_i A_{\alpha i}\frac{\partial }{\partial z_i}$ is a unitary frame at $z=0$,
 and $a_1, \ldots , a_n$ are nonnegative constants, with not all of them zero, then we have
$$ B_g(e,a) = \sum_{\alpha , \beta =1}^n  R_{e_{\alpha } \overline{e}_{\alpha } e_{\beta } \overline{e}_{\beta } } a_{\alpha } a_{\beta }
= - \left(\sum_{\alpha =1}^na_{\alpha } \right)^2 + (2-\eps )
\sum_{\alpha =1}^n a_{\alpha }^2. $$ If we take $a_1=\cdots
=a_n=\frac{1}{\sqrt{n}}$, then we get $B_g(e,a)=-n+2-\eps <0$, since
$n\geq 2$. So at the origin,
 hence in a neighborhood of the origin, the metric $g$ has positive holomorphic sectional curvature, but its real bisectional curvature is not even nonnegative.

Note that locally, if we take a metric $h$ that is given by the matrix $\ ^t\! (g_{i\overline{j}})^{-1}$,
 then the curvature of $h$ is that of $g$ with opposite sign. In particular, if we take the inverse transpose of the metric in Example \ref{Example2.2}, namely, if we let
$$ h_{i\overline{j}} = \frac{1}{1+|z|^2}\delta_{ij} + \frac{2-\eps} {(1+|z|^2)(1-(1-\epsilon )|z|^2) } z_i \overline{z}_j ,$$
then it would have $H<0$ near the origin, but $B\nleq 0$. This
completes the proof of Proposition \ref{proposition1.3}. \qed

Next, we would like to point out that for a Hermitian manifold, although the sign of the holomorphic sectional curvature $H$ does not
 control that of the real bisectional curvature $B$, the two are not too far apart from each other. For instance, the sign of $B$ does
  not control the sign of any of the three Ricci curvature tensors of the Chern connection. We have the following

\begin{example}
Consider a small ball  $D$ in ${\mathbb C}^2$ centered at the origin, equipped with the  Hermitian metric $g$ defined by
\begin{eqnarray*}
g_{1\overline{1}} & =  & 1 - |z_1|^2 + (1+b )|z_2|^2 \\
g_{2\overline{2}} & =  & 1 - (1+4b ) |z_1|^2 -|z_2|^2 \\
g_{1\overline{2}} & =  &  (1+b )z_2\overline{z}_1
\end{eqnarray*}
where $(z_1, z_2)$ is the Euclidean coordinate of ${\mathbb C}^2$ and $b >0$ is a constant.
 We claim that the metric has positive real bisectional curvature, but its first, second, and third Ricci curvature are not nonnegative.
\end{example}

At the origin $z=0$, $g_{i\overline{j}} = \delta_{ij}$, and $dg=0$, so the Chern curvature tensor
$$ R_{i\overline{j}k\overline{\ell }}  = - \frac{\partial ^2 g_{k\overline{\ell }} } {\partial z_i \partial \overline{z}_j }  $$
at $z=0$ has components
$$ R_{1\overline{1}1\overline{1}} = R_{2\overline{2}2\overline{2}} = 1, \ \  \ R_{1\overline{1}2\overline{2}}
 = R_{1\overline{2}2\overline{1}} = R_{2\overline{1}1\overline{2}} = -1-b, \ \ \ R_{2\overline{2}1\overline{1}} =1+4b, $$
and all other components are zero at $z=0$. Recall that the first, second, and third Ricci tensor of $R$ are defined by
$$ \mbox{Ric}^{(1)}_{i\overline{j}} = \sum_{k,\ell =1}^n g^{k\overline{\ell }} R_{i\overline{j}k\overline{\ell }} ,
 \ \ \ \mbox{Ric}^{(2)}_{i\overline{j}} = \sum_{k,\ell =1}^n g^{k\overline{\ell }} R_{k\overline{\ell }i\overline{j}},
  \ \ \ \mbox{Ric}^{(3)}_{i\overline{j}} = \sum_{k,\ell =1}^n g^{k\overline{\ell }} R_{i\overline{\ell }k\overline{j }}. $$
At the origin, we have
\begin{eqnarray*}
\mbox{Ric}^{(1)}_{1\overline{1}} & = & R_{1\overline{1}1\overline{1}} + R_{1\overline{1}2\overline{2}} \ = \ -b \\
\mbox{Ric}^{(2)}_{2\overline{2}} & = & R_{1\overline{1}2\overline{2}} + R_{2\overline{2}2\overline{2}} \ = \ -b \\
\mbox{Ric}^{(3)}_{1\overline{1}} & = &
R_{1\overline{1}1\overline{1}} + R_{1\overline{2}2\overline{1}} \ =
\ -b.
\end{eqnarray*}
So none of the three Ricci curvatures is positive at the origin. On the other hand, we claim that the real bisectional curvature of
 $g$ is positive at (hence near) the origin. That is, we want to show that
$$ B= \sum_{i,j,k,\ell } R_{i\overline{j}k\overline{\ell }} \ \xi_{ij} \xi_{k\ell } > 0 $$
for any non-trivial, non-negative Hermitian matrix $\xi = (\xi_{ij})$. Write $\xi_{11}=x$, $\xi_{22}=y$,
 and $\xi_{12}=t$, then we have $x\geq 0$, $y\geq 0$, $|t|^2\leq xy$, and either $x$ or $y$ is positive. At the origin, we have
\begin{eqnarray*}
B & = & x^2 + y^2 + 3b xy - 2(1+b) |t|^2 \\
& \geq & x^2 + y^2 + 3bxy -2(1+b) xy \\
& = & (x-y)^2 + bxy.
\end{eqnarray*}
Clearly, $B>0$ since either $x$ or $y$ will be positive. So the real bisectional curvature of $g$ is
 positive in a small neighborhood of the origin, yet each of the three Ricci tensors is not even non-negative.
  This shows that the sign of the real bisectional curvature does not control the sign of any of the three Ricci tensors.

\vs

\section{Manifolds with constant real bisectional curvature}

In this section, we will prove Theorem \ref{Theorem1.4} stated in
the introduction. Let $(M^n,g)$ be a compact Hermitian manifold with
constant
 real bisectional curvature $c$, and let $e$ be any unitary frame, then by the definition of real bisectional curvature $B_g$, we have
\begin{equation*}
\sum_{i,j,k,\ell } R_{i\overline{j}k\overline{\ell}} \xi_{ij} \xi_{k\ell } = c \sum_{i} \xi_{ii}^2
\end{equation*}
for any  non-trivial, nonnegative Hermitian $n\times n$ matrix $\xi$. This implies that
\begin{equation}
R_{i\overline{j}k\overline{\ell}} + R_{k\overline{\ell}i\overline{j}} = \left\{ \begin{array}{ll} 2c, \ \ \mbox{if} \ i=\ell\ \text{and}\ k=j; \\
\ 0, \  \ \mbox{otherwise}. \end{array} \right.\label{4}
\end{equation}

Next, let us follow the notations of \cite{YZ} and denote by $\{ \varphi_1, \ldots , \varphi_n\}$ the coframe of $(1,0)$-forms dual to $e$,
 with $\omega$ the K\"ahler form of $g$, and $\tau $ the column vector of torsion $(2,0)$-forms of the Chern connection. Denote by $\eta $ the Gauduchon $1$-form, then we have
$$ \tau_k = \sum_{i,j} T_{ij}^k\varphi_i \wedge \varphi_j, \ \ \ \ \eta = \sum_j \eta_j \varphi_j = \sum_{i,j} T^i_{ij}\varphi_j. $$
From Lemma 7 of \cite{YZ}, we have
\begin{equation}
2 T^k_{ij,\overline{\ell }} = R_{j\overline{\ell}i\overline{k}} - R_{i\overline{\ell}j\overline{k}}
\end{equation}
for any indices $1\leq i, j, k, \ell \leq n$, where the index after the comma stands for covariant differentiation with respect to the Chern
 connection. By letting $k=i$ and sum over, we get
 \begin{equation}
2  \ \eta_{j,\overline{\ell }} = \sum_k
(R_{j\overline{\ell}k\overline{k}} -
R_{k\overline{\ell}j\overline{k}}).\label{6}
\end{equation}
By formula (15) of \cite{YZ}, we have $\partial (\omega^{n-1})= -2 \eta \wedge \omega^{n-1}$, hence
\begin{equation}
\partial \overline{\partial } (\omega^{n-1}) = 2\overline{\partial }\eta \wedge \omega^{n-1} + 4 \eta \wedge \overline{\eta } \wedge
\omega^{n-1}.
\end{equation}
Integrating it over the compact manifold $M^n$, we get
\begin{equation}
\int_M \left(\sum_i \eta_{i,\overline{i}} \right) \ \omega^n = 2
\int_M |\eta|^2 \omega^n.\label{8}
\end{equation}
From (\ref{6}), we get
\begin{eqnarray*}
2 \sum_i \eta_{i,\overline{i}}  & = & \sum_{i,k}
(R_{i\overline{i}k\overline{k}} -  R_{k\overline{i}i\overline{k}})
\ = \  \sum_{i\neq k} (R_{i\overline{i}k\overline{k}} -  R_{k\overline{i}i\overline{k}})\\
& = &  \sum_{i\neq k} (- R_{k\overline{k}i\overline{i}} +
R_{i\overline{k}k\overline{i}}-2c) \ = \ - 2 \sum_i
\eta_{i,\overline{i}}-2cn(n-1)
\end{eqnarray*}
where we used (\ref{4}) in the first equality of the second line.
So
$$\sum_i \eta_{i,\overline{i}}=-\frac{1}{2}cn(n-1)$$ everywhere on
$M$. By (\ref{8}), we know that $c\leq 0$. Moreover, when $c=0$, we
have $\eta=0$, and the Hermitian manifold $(M^n,g)$ is balanced, i.e
$d\omega^{n-1}=0$.

Next let us focus on the $c=0$ case. Recall that the first, second,
and third Ricci curvature tensor of the Chern curvature tensor $R$
are defined by
\begin{equation} \mbox{Ric}^{(1)}_{i\overline{j}} = \sum_k R_{i\overline{j}k\overline{k}} \ , \ \ \ \ \
 \mbox{Ric}^{(2)}_{i\overline{j}} = \sum_k R_{k\overline{k}i\overline{j}} \ , \ \ \  \ \ \mbox{Ric}^{(3)}_{i\overline{j}} = \sum_k
 R_{i\overline{k}k\overline{j}}.
\end{equation}
By the fact that $\eta =0$ and (\ref{6}), we know that
\begin{equation} \sum_k R_{i\overline{j} k\overline{k}}  = \sum_k
R_{k\overline{j}i\overline{k}}\label{10}
\end{equation}
for any $1\leq i,j\leq n$ and
$$ \sum_k R_{i\overline{j} k\overline{k}}  = \sum_k R_{k\overline{j}i\overline{k}}
 = - \sum_k R_{i\overline{k}k\overline{j}} = - \sum_k \overline{R_{k\overline{i}j\overline{k}}}
  = - \sum_k \overline{R_{j\overline{i}k\overline{k}}} $$
where the first and last equalities are by formula (\ref{10}), and
the second equality is by (\ref{4}). This means that
$\mbox{Ric}^{(1)}_{i\overline{j}} =0$ for all $1\leq i,j\leq n$.
 Also, by (\ref{4}), we know that the Chern curvature tensor satisfies the skew-symmetry
\begin{equation} R_{x\overline{y}u\overline{v}} = - R_{u\overline{v}
x\overline{y}}\label{12}
\end{equation}
for any type $(1,0)$ tangent vectors $x$, $y$, $u$, and $v$. By
(\ref{10}), we know that $\mbox{Ric}^{(3)}=0$. By (\ref{12}), we get
that
$$ \mbox{Ric}^{(2)} = - \mbox{Ric}^{(1)} =0 .$$
This completes the proof of Theorem \ref{Theorem1.4} stated in \S 1.
\qed
\begin{remark}\begin{itemize}

\item[(a)] For the computations in this section, see also
\cite[Corollary~4.2, Corollary~4.5]{LiuYang} and
\cite[Theorem~3.1]{Yang14}.
\item[(b)]
We conjecture that a compact Hermitian manifold with vanishing real bisectional curvature must have vanishing Chern curvature,
 thus are compact quotients of complex Lie groups equipped with left invariant metrics.
\end{itemize}
\end{remark}

\section{The Hermitian form of Wu-Yau's Schwarz lemma}

The following formula is known as the Schwarz calculation
(e.g. \cite{Lu}, \cite{YauS}), and we include a slightly simpler proof here for
the readers' convenience.

\begin{lemma}\label{Lemma4.1}
Let $f:(M,g)\rightarrow (N,h)$ be a holomorphic map between
Hermitian manifolds. Then in the local holomorphic coordinates
$\{z_i\}$ and $\{w_\alpha\}$ on $M$ and $N$, respectively, we have
the identity
\begin{equation}
\Box_g u = |\nabla df|^2 + \left(g^{i\bar j} R^g_{i\bar j k\bar \ell} \right)
g^{k\bar q} g^{p\bar \ell} h_{\alpha\bar\beta} f^{\alpha }_p \overline{f^{\beta}_q}  - R^h_{\alpha\bar\beta\gamma\bar\delta}\left(g^{i\bar j}
 f^{\alpha}_i \overline{f^{\beta}_j} \right)  \left(g^{p\bar q}  f^{\gamma}_p \overline{
f^{\delta}_q} \right).
\label{chen-lu}
\end{equation}
where $u=\mbox{tr}_{\omega_g} (f^*\omega_h ) $, $f^{\alpha }_i = \frac{\partial f_{\alpha } } {\partial z_i}$, where the map $f$ is
 represented by $w_{\alpha }=f_{\alpha }(z)$ locally, $\nabla$ is the induced
connection on the bundle $E=T^*_M\otimes f^*(T_N))$, and $\Box_g u $ $= \mbox{tr}_{\omega_g} (\sqrt{-1}\partial \overline{\partial }u ) $
 is the complex Laplacian of $u$.
\end{lemma}

\begin{proof}
Let
$\displaystyle s = \p f=f^{\alpha}_i dz_i\otimes  e_\alpha  \in
\Gamma(M,E)$, where $e_{\alpha} = f^{\ast }\frac{\partial }{\partial w_{\alpha}}$. Since $f$ is a holomorphic map, $s$ is a holomorphic
section of $E$, i.e. $\bp s=0$. Thus by Bochner's formula, we have
 $$ \p \bp |s|^2 = \langle \nabla's, \nabla's \rangle - \langle\Theta^E s,s\rangle $$
where $\Theta^E$ is the curvature of the  vector bundle $E$ with
respect to the induced metric. More precisely, we have
$$ \Theta^E= \Theta^{T^*_M}\otimes  \text{Id}_{f^*(T_N)}+\text{Id}_{T^{\ast }_M}\otimes f^*(\Theta^{T_N}).$$
By taking trace, we obtain
$$ \mbox{tr}_{\omega_g} (\sq\p\bp|s|^2)=|\nabla's|^2-\langle \mbox{tr}_{\omega_g}
\sqrt{-1} \Theta^E s,s\rangle $$
 which is exactly the formula
(\ref{chen-lu}) since $|s|^2= \mbox{tr}_{\omega_g} f^*\omega_h$.
\end{proof}
Recall that on a Hermitian manifold $(M,g)$, the curvature term
$g^{i\bar j} R^g_{i\bar j k\bar \ell}$ in (\ref{chen-lu}) is called
the \emph{second (Chern) Ricci curvature} and  is denoted by
$\Ric^{(2)}$. It is different from the classic (first Chern) Ricci
curvature tensor
$$   \mbox{Ric}^{(1)}_{k\overline{\ell}} = \sum_{i,j} g^{i\bar j} R^g_{k\bar \ell i\bar j }=-\frac{\p^2\log\det g}{\p
z^k\p\bar z^\ell}\ ,$$
although they coincide when $g$ is K\"ahler.  As an application of Lemma 4.1, we get

\begin{lemma}
Let $f:(M,g)\rightarrow (N ,h)$ be a holomorphic map
between two Hermitian manifolds.  Then outside the set of critical points
of $f$, one has
\begin{equation}
 \Box_g \log u \geq \frac{1}{u} \left[ R^{(2)}_{ k\bar \ell} g^{k\bar q}
g^{p\bar \ell}  h_{\alpha\bar\beta}    f^{\alpha}_p \overline{ f^{\beta}_q} -R^h_{\alpha\bar\beta\gamma\bar\delta}\left(g^{i\bar j}
f^{\alpha}_i \overline{f^{\beta}_j} \right)\left(g^{p\bar q} f^{\gamma}_p
\overline{f^{\delta}_q}\right)  \right] \label{chen-lu-kato}
\end{equation}
where $u=\mbox{tr}_{\omega_g} (f^*\omega_h ) $,  $\Box_g $ is the complex Laplacian, and $R^{(2)}_{k\bar\ell}= \mbox{Ric}^{(2)}_{k\bar\ell} $
 is the second Ricci curvature of the Hermitian manifold $(M,g)$.
\end{lemma}
\begin{proof}
By using $|df|^2=|\p f|^2=\mbox{tr}_{\omega_g} f^*\omega_h$ and the
formula (\ref{chen-lu}), we know that if $df\neq 0$,
\begin{eqnarray*}
 \Box_g \log |df|^2 &=& \frac{\Box_g |df|^2}{|df|^2} - \frac{|\p
|df|^2|^2}{|df|^4}\\
&=&\frac{\Box_g
|df|^2}{|df|^2}- \frac{4|\p|df||^2}{|df|^2}\\
&=&\frac{\Box_g
|df|^2}{|df|^2}- \frac{|\nabla|df||^2}{|df|^2}.
\end{eqnarray*}
By formula (\ref{chen-lu}), we have
\begin{eqnarray*}
 \Box_g\log |df|^2 &=& \frac{ R^{(2)}_{k\bar \ell} g^{k\bar q} g^{p\bar \ell} h_{\alpha\bar\beta} f^{\alpha}_p \overline{ f^{\beta}_q}
  -R^h_{\alpha\bar\beta\gamma\bar\delta}\left(g^{i\bar j} f^{\alpha}_i \overline{f^{\beta}_j} \right)\left(g^{p\bar q} f^{\gamma}_p
\overline{f^{\delta}_q} \right)}{|df|^2}\\
&&+\frac{|\nabla df|^2-\left|\nabla
|df|\right|^2}{|df|^2}.
\end{eqnarray*}
The well-known Kato's inequality (e.g.
\cite{CGH}) says that for any section $\xi$ of an abstract
Riemannian vector bundle $(E,\nabla)$, one has
$$ |\nabla
|\xi||\leq|\nabla\xi|, $$
outside the zero set of $\xi$. Hence, we get (\ref{chen-lu-kato}).
\end{proof}
By using formula (\ref{chen-lu-kato}), we obtain the
following refined version of Yau's Schwarz calculation on Hermitian
manifolds, which is also analogous to Royden's formulation
(\cite{Ro}) (see also \cite{WWY},\cite{WY}).

\begin{theorem}\label{schwarz2} Let $f: (M,g)\rightarrow (N,h)$ be non-constant holomorphic map
between two Hermitian manifolds.  Suppose that the second Ricci
 curvature of $g$ satisfies
  \begin{equation}
  \Ric^{(2)}(g)\geq  - \lambda \omega_g +\mu f^*\omega_h\label{2riccibound}
  \end{equation}
  for continuous functions $\lambda$, $\mu $ where $\mu\geq 0$, and the real bisectional curvature of $h$ is bounded from above by
a continuous function $-\kappa \leq 0$ on $N$. Then we have
\begin{equation}   \Box_g u \geq - \lambda
u + \left(\frac{f^{\ast }\kappa }{r}+\frac{\mu}{n}\right)u^2\label{max1}
\end{equation}
and outside the zero locus of $df$:
 \begin{equation} \Box_g \log u \geq - \lambda+\left(\frac{f^{\ast }\kappa }{r}+\frac{\mu}{n}\right) u \label{max}
 \end{equation}
where $u=|df|^2= \mbox{tr}_{\omega_g} (f^{\ast }\omega_h)$ and  $r$ is the maximal rank of $df$.
\end{theorem}

\begin{proof}
By formula (\ref{2riccibound}), we have
\begin{eqnarray}
R^{(2)}_{ k\bar \ell}  g^{k\bar q}  g^{p\bar \ell}
h_{\alpha\bar\beta}    f^{\alpha}_p \overline{f^{\beta}_q}
& \geq &  ( - \lambda g_{k\bar\ell} + \mu
h_{\gamma\bar\delta}     f^{\gamma}_k \overline{ f^{\delta}_{\ell} } )  \ g^{k\bar q}
g^{p\bar \ell}  h_{\alpha\bar\beta}   f^{\alpha}_p \overline{ f^{\beta}_q } \nonumber \\
& \geq & - \lambda \left( \mbox{tr}_{\omega_g}
f^*\omega_h \right) + \frac{\mu}{n}   \left(\mbox{tr}_{\omega_g}
f^*\omega_h\right)^2,  \label{11}
\end{eqnarray}
where in the last step
we used the fact $\mu\geq 0$ and the Cauchy-Schwarz inequality:
\begin{equation}
(h_{\gamma\bar\delta}  f^{\gamma}_k  \overline{ f^{\delta}_{\ell} } )\  g^{k\bar q}
g^{p\bar \ell} (h_{\alpha\bar\beta}   f^{\alpha}_p \overline{ f^{\beta}_q } ) \geq \frac{1}{n}\left(\mbox{tr}_{\omega_g} f^*\omega_h\right)^2.
\label{55}
\end{equation}
Indeed, if we set $g_{i\bar j}=\delta_{ij}$ and
denote by $\Phi_{p\bar q}=\sum_{\alpha,\beta}h_{\alpha\bar\beta} f^{\alpha}_p
 \overline{ f^{\beta}_q}$ , then (\ref{55}) is
equivalent to
\begin{equation}
\sum_{p,q} |\Phi_{p\bar q}|^2\geq \frac{1}{n}
 \left(\sum_p \Phi_{p\bar p} \right)^2,
\end{equation}
which is obviously true. Next, we
use ideas in \cite{Ro} to estimate the second term on the right hand
side of (\ref{chen-lu-kato}).

 At a fixed point $p\in M$,  by taking unitary changes of coordinates at $p$ and
$f(p)$, we may  assume that the matrix $\left[ f^{\alpha}_i \right]$ has the canonical form
\begin{equation}
f^{\alpha}_i=\lambda_i\delta_i^\alpha
\end{equation}
with $\lambda_1\geq \lambda_2\geq
\cdots\geq\lambda_r>\lambda_{r+1}=\cdots=0$, where $r$ is the
rank of the matrix $\left[f^{\alpha}_i \right]$. Then
$$\mbox{tr}_{\omega_g} f^*\omega_h=\sum_i \lambda_i^2=\sum_\alpha
\lambda^2_\alpha.$$
Hence, we have
\begin{eqnarray}
\nonumber && R^h_{\alpha\bar\beta\gamma\bar\delta}\left(g^{i\bar j}
     f^{\alpha}_i \overline{ f^{\beta}_j}  \right)\left(g^{p\bar q} f^{\gamma}_p \overline{ f^{\delta}_q}  \right) \ =
\sum_{\alpha,\beta,\gamma,\delta,i,k}R^h_{\alpha\bar\beta\gamma\bar\delta}\lambda_i^2\lambda_k^2
\cdot
\delta_i^\alpha\delta_i^\beta\delta_k^\gamma\delta_k^\delta\\
\label{key111111}&& = \sum_{\alpha,\gamma,i,k}R^h_{\alpha\bar\alpha\gamma\bar\gamma}\lambda_i^2\lambda_k^2
\cdot \delta_i^\alpha\delta_k^\gamma \ = \ \sum_{\alpha,\gamma}R^h_{\alpha\bar\alpha\gamma\bar\gamma}\lambda_\alpha^2\lambda_\gamma^2.
\end{eqnarray}
Since the real bisectional curvature of $(N, h)$ is bounded from
above by $-\kappa\leq 0$, we get
\begin{equation}
\sum_{\alpha,\gamma}R^h_{\alpha\bar\alpha\gamma\bar\gamma}\lambda_\alpha^2\lambda_\gamma^2\leq
-\kappa \left(\sum_\alpha \lambda_\alpha^4\right)\leq
-\frac{\kappa }{r}\left(\sum_\alpha \lambda_\alpha^2\right)^2=
-\frac{\kappa }{r}(\mbox{tr}_{\omega_g} f^*\omega_h)^2.\label{22}
\end{equation}
The last
inequality follows from the fact that $\kappa \geq 0$ and $r$ is the
maximal number of nonzero elements of $\lambda_\alpha$. Therefore,
by using formulas (\ref{chen-lu}), (\ref{11}) and (\ref{22}), we
obtain (\ref{max1}). By using (\ref{chen-lu-kato}), (\ref{11}) and
(\ref{22}), we obtain (\ref{max}).
\end{proof}

If we apply the above theorem to the identity map of $M$, we get the following:

\begin{corollary} \label{schwarz} Let $M$ be a compact complex manifold
with two Hermitian metrics $g$, $h$, such that $h$
has real bisectional curvature bounded above by a continuous function
$-\kappa \leq 0$ on $M$, and $g$ satisfies
\begin{equation}\label{assum}
\Ric^{(2)}(g)\geq -\lambda \omega_g+\mu\omega_h,
\end{equation}
for some continuous functions $\lambda$, $\mu$ with $\mu \geq 0$. Then we have
\begin{equation}
\Box_{g}\log u \geq \left(\frac{\kappa}{n}+\frac{\mu}{n}\right) u -\lambda , \label{key2}
\end{equation}
where $u=\mbox{tr}_{\omega_g} \omega_h$. In particular, if $\kappa$, $\lambda$, $\mu$ are all constants and $\kappa +\mu >0$, then
\begin{equation}\sup_M u \leq \frac{n\lambda}{\kappa+\mu}.
\end{equation}
\end{corollary}

As an immediate consequence of the above Schwarz calculation and
Yau's generalized maximum principle \cite{Yau75}, we get a proof of
Theorem \ref{Theorem1.10}.
 Since  the statement involves general Hermitian manifolds, let us give a detailed proof for the convenience of the readers. First, recall that
  Yau's maximum principle says that, on a complete Riemannian manifold $M$ with Ricci curvature bounded from below, if $v$ is a smooth function
   bounded from below, then for any $\eps >0$, there exists a point $p_{\eps}$ in $M$, such that at $p_{\eps}$,
\begin{equation}
|\nabla v | < \eps , \ \ \ \ \Delta v > -\eps , \ \ \ \ v(p_{\eps })
< \inf v + \eps .\label{27}
\end{equation}
Here $\Delta$ is the Laplacian and $\nabla$ the Levi-Civita connection on $M$.

Next, let us recall the relation between the Ricci curvatures of the Riemannian and Chern connections of a Hermitian manifold. Let $(M^n,g)$
 be a Hermitian manifold. It is well-known that
\begin{equation}
\Delta v = 2 \ \Box_g v + 2 \sum_{i=1}^n (v_i \overline{\eta_i} +
v_{\overline{i}} \eta_i )\label{28}
\end{equation}
where $\eta = \sum_i \eta_i \varphi_i$ is the Gauduchon $1$-form, $v_i = e_i(v)$, $v_{\overline{i}} = \overline{e}_i(v)$, and $e$
 is any unitary frame with $\varphi$ the dual coframe.

Now let us assume that $|\eta| \leq C$ on $M$ for some constant $C>0$, and let $u\geq 0$ be a smooth function on $M$ satisfying
 the inequality $$\Box_g u \geq -bu+au^2$$ for some constants $a$, $b$, with $a>0$.  Then we claim that Yau's maximum principle
  will imply that $u\leq \frac{b}{a}$.

To see this, note that by (\ref{28}) we have
\begin{equation}
\Delta u \geq -2b u + 2a u^2 - 4C |\nabla u|.\label{29}
\end{equation}
Let us consider the smooth, positive function $v=(u+1)^{-\frac{1}{2}}$ on $M$. We have $v'=-v^3/2$, $v'' = 3v^5/4$, and $\nabla v = v' \nabla u$, so
$$ \Delta v = -\frac{v^3}{2}\Delta u + \frac{3v^5}{4} |\nabla u|^2 =  -\frac{v^3}{2}\Delta u + \frac{3}{v}|\nabla v|^2 ,$$
and
$$ -\frac{2}{v^3}\Delta v + \frac{6}{v^4} |\nabla v|^2 = \Delta u > -bu + au^2 - 4C |\nabla u| $$
by (\ref{29}). That is, we always have
\begin{equation}
-bu + au^2 < -\frac{2}{v^3}\Delta v + \frac{6}{v^4} |\nabla v|^2 +
\frac{8C}{v^3} |\nabla v|.\label{30}
\end{equation}
Now if $(M^n,g)$ as a Riemannian manifold is complete and with Ricci curvature bounded from below, then by Yau's maximum principle,
 for any $\eps >0$, there will be $p_{\eps } \in M$ at which (\ref{27}) holds, so by (\ref{30}), at $p_{\eps }$ we have
\begin{equation}
 -bu + au^2 < (2+8C)\eps (u+1)^{\frac{3}{2}} + 6\eps^2 (u+1)^2.
\end{equation}
Since $a>0$ is a constant, by choosing $\eps$ sufficiently small, we know that $\sup u$ must be finite. When $\eps \rightarrow 0$,
 $u(p_{\eps }) \rightarrow \sup u$, so the above inequality gives $\sup u \leq \frac{b}{a}$. Applying Theorem \ref{schwarz2} in the
  case where $\mu =0$, both $\lambda =b$ and $\kappa =k >0$ are constants,  we get proved the following

\begin{theorem}\label{Theorem4.5}
Let $(M^n,g)$ be a Hermitian manifold with bounded Gauduchon $1$-form $\eta $, and its second (Chern) Ricci curvature
 is bounded from below by a constant $-b$, and as a Riemannian manifold it is complete and has Ricci curvature bounded
  from below. Let $(N^m,h)$ be a Hermitian manifold whose real bisectional curvature is bounded from above by a negative
   constant $-k<0$. If $f: M\rightarrow N$ is any non-constant holomorphic map, then we must have $b>0$, and
\begin{equation}
\sup_M \mbox{tr}_{\omega_g} f^{\ast } \omega_h \leq \frac{rb}{k}
\end{equation}
where $r$ is the maximum rank of $df$.
\end{theorem}

In particular, if we start with $b=0$ at the beginning, then we know
that $f$ must be constant. This gives a proof to Theorem
\ref{Theorem1.10}.

\begin{remark} There are also some Schwarz type inequalities in
\cite{Tosatti07} under different curvature assumptions.
\end{remark}

\section{Nonpositive and quasi-negative real bisectional curvature}
In this section we will give proofs to Theorems \ref{Theorem1.7}
through \ref{Theorem1.9} and \ref{Theorem1.11} stated in the
introduction. The proofs
 are basically the same as those given in \cite{WY}, \cite{TY}, and \cite{DT}, with the simple fact that, all
  metrics are equivalent on a compact manifold.

\noindent {\bf Proof of Theorem \ref{Theorem1.7}:} Let $(M^n,h)$ be
a compact Hermitian manifold with nonpositive real bisectional
 curvature $B_h$. Let $g$ be a K\"ahler metric on $M$. Assuming that the canonical line bundle $K_M$ is not nef, and
  we want to derive a contradiction as in the proof of Theorem 1.1 in \cite{TY}.

Following \cite{TY}, first of all, since $K_M$ is not nef, there will be $\eps_0>0$ such that $\eps_0[\omega_g ] - c_1(M)$ is
 nef but not a K\"ahler class. Then for any $\eps >0$, the class $(\eps_0+\eps )[\omega_g] -c_1(M)$ is a K\"ahler class.
  Write $\omega$ for $\omega_g$, and denote by $\mbox{Ric}(g)=-\sqrt{-1} \partial \overline{\partial} \log \omega^n$ the $(1,1)$
   form of the first Chern Ricci of the $g$, which represents $c_1(M)$, thus there exists smooth function $\varphi_{\eps}$
    and $\psi_{\eps}$ on $M$ such that the K\"ahler metric
$$ \omega_{\eps} : = (\eps_0+\eps )\omega - \mbox{Ric}(g) + \sqrt{-1} \partial \overline{\partial} u_{\eps} , $$
where $u_{\eps} = \varphi_{\eps}+\psi_{\eps}$, satisfies $ \omega_{\eps}^n = e^{u_{\eps}}\omega^n$, or equivalently,
$$  \mbox{Ric} (\omega_{\eps}) = \mbox{Ric}(g) - \sqrt{-1} \partial \overline{\partial} u_{\eps}  = - \omega_{\eps} + (\eps_0+\eps )\omega .$$
Since $M$ is compact, there exists a constant $D>0$ such that $\frac{1}{D} \omega \leq \omega_h\leq D\omega$. So one has
$$  \mbox{Ric} (\omega_{\eps}) \geq - \omega_{\eps} + \frac{\eps_0+\eps}{D}\omega_h .$$
Now if we apply the Schwarz Lemma( Corollary \ref{schwarz}) to the
identity map from $(M^n, \omega_{\eps})$ onto $(M^n,h)$, with
$\lambda =1$,
 $\mu = \frac{\eps_0+\eps}{D}$, and $\kappa =0$, then we get
$$ \mbox{tr}_{\omega_{\eps}} \omega_h \leq \frac{nD}{\eps_0+\eps } $$
as in \cite{TY}. So $\mbox{tr}_{\omega_{\eps}} \omega \leq \frac{nD^2}{\eps_0+\eps }$. The rest of the argument is identical
 to that of \cite{TY}, so Theorem \ref{Theorem1.7} holds. \qed

Next, let us prove Theorem \ref{Theorem1.8}. Again the same proof of
\cite{DT} can be slightly modified to cover this case.

\noindent {\bf Proof of Theorem \ref{Theorem1.8}:} Suppose that
$(M^n,h)$ be a compact Hermitian manifold with quasi-negative real
 bisectional curvature $B_h$. Let $g$ be a K\"ahler metric on $M$ and write $\omega$ for $\omega_g$. The canonical line
  bundle $K_M$ is nef by Theorem \ref{Theorem1.7}. As observed in \cite{DT}, it suffices to show
\begin{equation}
c_1^n(K_M) >0 ,
\end{equation}
as then $K_M$ will be big by a result of Demailly and Paun
\cite{DP}. Thus $M$ will be Moishezon, hence projective as it is
assumed to be K\"ahler. Now the Kawamata Theorem implies that $K_M$
must be ample since it does not contain any rational curves.

Since $K_M$ is nef, by  \cite[Proposition ~8]{WY}, for any $\eps
>0$, there exists smooth function $u_{\eps}$ on $M$ such that the
K\"ahler metric
$$ \omega_{\eps} := \eps \omega - \mbox{Ric}(g) + \sqrt{-1}\partial \overline{\partial} u_{\eps} $$
satisfies $\omega_{\eps }^n = e^{u_{\eps }} \omega^n$, or equivalently,
\begin{equation}
\mbox{Ric}(\omega_{\eps} ) = \mbox{Ric}(g) - \sqrt{-1}\partial \overline{\partial} u_{\eps} = - \omega_{\eps} + \eps \omega,
\end{equation}
and $u_{\eps }\leq C$ for such constant $C$ independent of $\eps$.
The same proof of the inequality $c_1^n(K_M)>0$  in \cite{DT} will go through provided that, for each $\eps >0$, there will be smooth function $S_{\eps }>0$ on $M$ such that
\begin{equation}
\Delta_{\omega_{\eps}} \log S_{\eps} \geq \frac{n+1}{2n}\kappa S_{\eps} -1
\end{equation}
holds, where $\kappa$ is a continuous function on $M$ which is quasi-positive, namely, non-negative everywhere and positive somewhere.

Now if we consider the Hermitian metric $h$ on $M$ with quasi-negative real bisectional curvature. We may let $\kappa $ be a smooth
 quasi-positive function on $M$ such that $B_h\leq - \kappa$. Then by applying Corollary \ref{schwarz}, which is the Hermitian version
 of \cite[ Proposition ~9]{WY}, to the identity map from $(M, \omega_{\eps})$ onto $(M, h)$, we get the above inequality for the
   function $S_{\eps } = \mbox{tr}_{\omega_{\eps}} \omega_h $, since we have
$$  \mbox{Ric}(\omega_{\eps} ) = - \omega_{\eps} + \eps \omega \geq - \omega_{\eps} + \frac{\eps}{D} \omega_h ,$$
where $D>0$ is a constant such that $\frac{1}{D}\omega \leq \omega_h
\leq D \omega$ as before. This completes the proof of Theorem
\ref{Theorem1.8}. \qed

Next, let us prove Theorem \ref{Theorem1.9}, which is a rather
special case, but without the assumption that $M$ is K\"ahlerian a
priori.

\noindent {\bf Proof of Theorem \ref{Theorem1.9}:} Let $M$ and $N$
be as in Theorem \ref{Theorem1.9}. Assume the contrary that there is
a bimeromorphic map
 $f$ from $N$ into $M$. Since $M$ is compact and with nonpositive holomorphic sectional curvature, by the result of Shiffman and Griffiths,
  $M$ obeys the Hartog's phenomenon. So any meromorphic map into $M$ must be holomorphic. Let $U\subseteq M$ be the open set where real
   bisectional curvature is negative. Let $Y$ be a generic fiber of $N$ such that the restriction map $f|_Y: Y\rightarrow M$ is
    non-constant and its image intersects $U$. By assumption, $Y$ is a compact K\"ahler manifold with $c_1=0$, thus admits a Ricci-flat
     K\"ahler metric by Yau's solution to the Calabi conjecture. By applying the Schwarz Lemma Theorem \ref{Theorem4.5} to $f|_Y$, we get a contradiction.
      So $M$ cannot be bimeromorphic to $N$. \qed

Finally, we prove Theorem \ref{Theorem1.11}, which is a direct
consequence of Lemma \ref{Lemma4.1}.

\noindent {\bf Proof of Theorem \ref{Theorem1.11}:} Let $f: (M,g)
\rightarrow (N,h)$ be a non-constant holomorphic map between
Hermitian manifolds,
 with $M$ being compact. By the curvature assumptions on $M$ and $N$, and Lemma \ref{Lemma4.1}, we get
$$ \Box_g u \geq |\nabla df |^2  $$
where $u=|df|^2= \mbox{tr}_{\omega_g} f^{\ast }\omega_h$. By Gauduchon's theorem, there exists a smooth function $v$ on $M$
 such that $\partial \overline{\partial }(e^v \omega_g^{n-1})=0$. So we get
$$ \int_M |\nabla df |^2 e^v \omega_g^n \leq \int_M \Box_g u \ e^v \omega_g^n = \int_M n\sqrt{-1} \partial \overline{\partial} u \wedge  e^v \omega_g^{n-1} = 0 ,$$
thus $\nabla df =0$ everywhere on $M$. That is, $f$ is totally
geodesic. This completes the proof of Theorem \ref{Theorem1.11}.
\qed

\vskip 1\baselineskip

\end{document}